\documentclass[12pt,reqno]{amsart}

\usepackage{multirow}
\usepackage{hhline}

\usepackage{mathtools}
\usepackage{times}
\usepackage[T1]{fontenc}
\usepackage{mathrsfs}
\usepackage{latexsym}
\usepackage[dvips]{graphics}
\usepackage[titletoc, title]{appendix}
\usepackage{amsmath,amsfonts,amsthm,amssymb,amscd}
\usepackage{color}
\usepackage{hyperref}
\usepackage{amsmath}

\usepackage{color}
\usepackage{breakurl}

\usepackage{comment}
\newcommand{\bburl}[1]{\textcolor{blue}{\url{#1}}}

\newtheorem{thm}{Theorem}[section]

\newtheorem{cor}[thm]{Corollary}
\newtheorem{claim}[thm]{Claim}

\newtheorem{prop}[thm]{Proposition}

\newtheorem{defi}[thm]{Definition}
\newtheorem{rek}[thm]{Remark}

\DeclareMathOperator{\supp}{supp}
\DeclareMathOperator{\spann}{span}
\DeclareMathOperator{\sgn}{sgn}
\numberwithin{equation}{section}

\DeclareFontFamily{U}{mathx}{}
\DeclareFontShape{U}{mathx}{m}{n}{<-> mathx10}{}
\DeclareSymbolFont{mathx}{U}{mathx}{m}{n}
\DeclareMathAccent{\widehat}{0}{mathx}{"70}
\DeclareMathAccent{\widecheck}{0}{mathx}{"71}

\begin{document}

\title{Larger greedy sums for reverse partially greedy bases}

\author{H\`ung Vi\d{\^e}t Chu}

\email{\textcolor{blue}{\href{mailto:hungchu1@tamu.edu}{hungchu1@tamu.edu}}}
\address{Department of Mathematics, Texas A\&M University, College Station, TX 77843, USA}

\begin{abstract} 

An interesting result due to Dilworth et al. was that if we enlarge greedy sums by a constant factor $\lambda > 1$ in the condition defining the greedy property, then we obtain an equivalence of the almost greedy property, a strictly weaker property. 
Previously, the author of the present paper showed that 
enlarging greedy sums by $\lambda$ in the condition defining the partially greedy (PG) property also strictly weakens the property.
However, enlarging greedy sums in the definition of reverse partially greedy (RPG) bases by Dilworth and Khurana again gives RPG bases. The companion of PG and RPG bases suggests the existence of a characterization of RPG bases which, when greedy sums are enlarged, gives an analog of a result that holds for partially greedy bases. 
In this paper, we show that such a characterization indeed exists, answering positively a question previously posed by the author. 
\end{abstract}

\subjclass[2020]{41A65; 46B15}

\keywords{reverse partially greedy; greedy sum; bases.}

\thanks{The author is thankful to Timur Oikhberg for helpful comments on an earlier draft of this paper. The author would also like to thank the anonymous referees for a careful reading and constructive feedback that improves the paper's exposition.}

\maketitle

\tableofcontents

\section{Introduction}
Let $\mathbb{X}$ be an infinite-dimensional Banach space (with dual $\mathbb{X}^*$) over the field $\mathbb K = \mathbb R$ or $\mathbb C$. We define a \textbf{basis} to be any countable collection $\mathcal{B} = (e_n)_{n=1}^\infty$ such that i) the span of $(e_n)_n$ is norm-dense in $\mathbb{X}$; ii) there exist biorthogonal functionals $(e^*_n)_n\subset \mathbb{X}^*$ such that $e^*_{n}(e_m) = \delta_{n,m}$; and iii) there exist $c_1, c_2> 0$ such that $c_1\leqslant \|e_n\|, \|e_n^*\|\leqslant c_2$ for all $n$. In the literature, the totality condition: $\overline{\spann(e^*_n)}^{w^*} = X^*$ is sometimes called for to guarantee that for each $x$, the sequence $(e_n^*(x))_{n=1}^{\infty}$ is unique. Since uniqueness is not important for our purpose, we do not assume totality.

In 1999, Konyagin and Temlyakov \cite{KT1} introduced the notion of greedy bases as follows: for each $x\in\mathbb{X}$, a finite set $\Lambda$ is a greedy set of $x$ if 
$$\min_{n\in \Lambda}|e_n^*(x)| \ \geqslant\ \max_{n\notin \Lambda}|e_n^*(x)|.$$
For $m\in \mathbb{N}$, let $G(x, m)$ denote the set of all greedy sets of $x$ of cardinality $m$. A basis is said to be \textbf{greedy} if there exists a constant $\mathbf C\geqslant 1$ such that 
\begin{equation}\label{e8}\|x-P_{\Lambda}(x)\|\ \leqslant\ \mathbf C\sigma_m(x), \forall x\in\mathbb{X}, \forall m\in\mathbb{N}, \forall \Lambda\in G(x, m),\end{equation}
where $P_A(x) = \sum_{n\in A}e_n^*(x)e_n$ for any finite set $A$ and 
$$\sigma_m(x)\ :=\ \inf\left\{\left\|x-\sum_{n\in A}a_ne_n\right\|\,:\, a_n\in\mathbb{K}, |A|\leqslant m\right\}.$$
Later, Dilworth et al. \cite{DKKT} introduced the so-called \textbf{almost greedy} bases: 
a basis is almost greedy if there exists a constant $\mathbf C\geqslant 1$ such that
\begin{equation}\label{e9}\|x-P_{\Lambda}(x)\|\ \leqslant\ \mathbf C\inf\{\|x-P_A(x)\|: |A| = m\}, \forall x\in \mathbb{X}, \forall m\in\mathbb{N}, \forall \Lambda\in G(x,m).\end{equation}
By definition, a greedy basis is almost greedy, but an almost greedy basis is not necessarily greedy (see \cite{KT1} or \cite[Example 10.2.9]{AK}). 
Among other results, Dilworth et al. showed a suprising equivalence of almost greedy bases. In particular, for any fixed $\lambda > 1$, if we enlarge greedy sums from size $m$ to $\lceil \lambda m\rceil$ in \eqref{e8} to have
\begin{equation}\label{e10}\|x-P_{\Lambda}(x)\|\ \leqslant\ \mathbf C\sigma_m(x), \forall x\in\mathbb{X}, \forall m\in\mathbb{N}, \forall \Lambda\in G(x, \lceil \lambda m\rceil),\end{equation}
then the condition \eqref{e10} is equivalent to the condition \eqref{e9} (with possibly different constants). That is, by increasing the size of greedy sums linearly, we move from the realm of greedy bases to a strictly weaker realm of almost greedy bases. From another perspective, for an almost greedy basis, enlarged greedy sums are optimal. 

Continuing the work, the author of the present paper \cite{C1} investigated the situation for \textbf{partially greedy} (PG) bases (also introduced by Dilworth et al. for Schauder bases \cite{DKKT} and by Berasategui et al. for general bases \cite{BBL}), which are defined to satisfy the condition: there exists a constant $\mathbf C\geqslant 1$ such that
\begin{equation}\label{e11}\|x-P_{\Lambda}(x)\|\ \leqslant\ \mathbf C\inf_{k\leqslant m}\left\|x-\sum_{n=1}^k e_n^*(x)e_n\right\|, \forall x\in \mathbb{X}, \forall m\in\mathbb{N}, \forall \Lambda\in G(x,m).\end{equation}
Since their introduction in \cite{DKKT}, partially greedy bases have been studied extensively in \cite{BBL, BBC1, C1, C2, DK, DKO, K}. According to \cite[Theorem 1.8]{C1}, if we enlarge the size of greedy sums from $m$ to $\lceil\lambda m\rceil$ in \eqref{e11}, we obtain strictly weaker greedy-type bases. Therefore, similar to Dilworth et al.'s result, enlarging greedy sums in the PG property strictly weakens the property. 

To complete the picture, let us consider \textbf{reverse partially greedy} (RPG) bases introduced by Dilworth and Khurana \cite{DK}. 
The original definition of RPG \cite{DK} is relatively more technical: first, given two sets $A, B\subset\mathbb{N}$, we write $A > B$ if for all $a\in A$ and $b\in B$, we have $a > b$. Respective definitions hold for other inequalities $<, \geqslant$, and $\leqslant$. Also, it holds vacuously that $\emptyset > A$ and $\emptyset < A$ for all sets $A$. 
A basis is RPG if there exists a constant $\mathbf C\geqslant 1$ such that
\begin{equation}\label{e12}\|x-P_{\Lambda}(x)\|\ \leqslant\ \mathbf C\widetilde{\sigma}^{R, \Lambda}_m(x), \forall x\in\mathbb{X}, \forall m\in\mathbb{N}, \forall \Lambda\in G(x,m),\end{equation}
where $$\widetilde{\sigma}^{R,\Lambda}_m(x)\ =\ \inf\{\|x-P_A(x)\|\,:\, |A|\leqslant m, A > \Lambda\}.$$
Recent work has confirmed that RPG bases are truly companions of PG bases (see \cite{C1, DK, K}). That is, if a result holds true for PG bases, there is a corresponding result that holds for RPG bases. We, therefore, suspect that as in the case of PG bases, if we enlarge the size of greedy sums in a condition defining RPG bases, we would then obtain a strictly weaker greedy-type condition. However, by \cite[Theorem 5.8]{C1}, enlarging greedy sums in \eqref{e12} still gives us RPG bases. This hints us at the existence of an equivalent reformulation of RPG bases such that when we enlarge greedy sums in the reformulation, we strictly weaken the RPG property. The main results in this paper show that such an equivalent reformulation indeed exists. 

Given a set $A\subset\mathbb{N}$, define 
$$\widecheck{\sigma}^{A}_m(x)\ :=\ \inf\{\|x-P_{I}(x)\|\, :\, \mbox{either } I = \emptyset\mbox{ or } I\mbox{ is an interval},  A \leqslant \max I, |I|\leqslant m\}.$$
Our first result gives an equivalence of the notion of RPG bases.
\begin{thm}\label{m2}
A basis $\mathcal{B}$ is RPG if and only if there exists $\mathbf C \geqslant 1$ such that 
\begin{equation}\label{e1}\|x-P_{\Lambda}(x)\|\ \leqslant\ \mathbf C\widecheck{\sigma}^{\Lambda}_m(x),\forall x\in\mathbb{X}, \forall m\in\mathbb{N}, \forall \Lambda\in G(x, m).\end{equation}
\end{thm}
To the author's knowledge, \eqref{e1} is the first characterization of RPG bases such that the right side of \eqref{e1} involves a projection onto consecutive elements of the basis. This resembles \eqref{e11}, which defines PG bases. Indeed, \eqref{e1} highlights the correspondence between PG bases and RPG bases by revealing the relative position of the interval $I$ with respect to the greedy set $\Lambda$. Specifically, for PG bases, since we project onto the first elements of $\mathcal{B}$, we have $\min I \leqslant \Lambda$, while for RPG bases, $\Lambda\leqslant \max I$. This leads us to a new characterization of PG bases. 

\begin{thm}\label{ma1}
A basis $\mathcal{B}$ is PG if and only if there exists $\mathbf C \geqslant 1$ such that 
\begin{equation*}\|x-P_{\Lambda}(x)\|\ \leqslant\ \mathbf C\widehat{\sigma}^{\Lambda}_m(x),\forall x\in\mathbb{X}, \forall m\in\mathbb{N}, \forall \Lambda\in G(x, m),\end{equation*}
where $$\widehat{\sigma}^{\Lambda}_m(x) \ :=\ \inf\{\|x-P_{I}(x)\|\, :\, \mbox{either } I = \emptyset\mbox{ or } I\mbox{ is an interval},  \Lambda \geqslant \min I, |I|\leqslant m\}.$$
\end{thm}

Next, for any fixed $\lambda \geqslant 1$, we study the following condition: there exists $\mathbf C\geqslant 1$ such that
\begin{equation}\label{e5}
    \|x-P_{\Lambda}(x)\|\ \leqslant\ \mathbf C\widecheck{\sigma}_m^{\Lambda}(x), \forall x\in\mathbb{X}, \forall m\in \mathbb{N},\forall \Lambda\in G(x, \lceil \lambda m\rceil).
\end{equation}
As stated above, our goal is to show that \eqref{e5} is strictly weaker than the RPG property to have an analog of Dilworth et al.'s theorem for RPG bases. First, we give bases that satisfy \eqref{e5} a name. 

\begin{defi}\normalfont\label{bdf}
For a fixed $\lambda \geqslant 1$, a basis $\mathcal{B}$ is said to be $\lambda$-RPG of type 2 (or $\lambda$-RPG2, for short) if $\mathcal{B}$ satisfies \eqref{e5}. In this case, the least constant in \eqref{e5} is denoted by $\mathbf C_{\lambda, rp2}$, and the basis is said to be $\lambda$-RPG2 with constant $\mathbf C_{\lambda, rp2}$. We will use the shorthand ``w.const" in place of ``with constant". 
\end{defi}

Here we use ``of type 2" to not confuse ourselves with $\lambda$-RPG bases in \cite[Definition 5.5]{C1}, which defines a basis to be $\lambda$-RPG if it satisfies \eqref{e12} with enlarged greedy sums from $m$ to $\lceil\lambda m\rceil$. For all $\lambda\geqslant 1$, the $\lambda$-RPG property in \cite[Definition 5.5]{C1} is equivalent to RPG property (see \cite[Theorem 5.8]{C1}). We shall show that this is not the case for the $\lambda$-RPG2 property. 

\begin{thm}\label{m8}
Let $\lambda > 1$. The following hold.
\begin{enumerate}
    \item[i)] If $\mathcal{B}$ is RPG, then $\mathcal{B}$ is $\lambda$-RPG2.
    \item[ii)] There exists an unconditional basis $\mathcal{B}$ that is $\lambda$-RPG2 but is not RPG. 
\end{enumerate}
\end{thm}

To prove Theorem \ref{m8}, we need to characterize $\lambda$-RPG2 bases with the introduction of the so-called reverse partial symmetry for largest coefficients (Definition \ref{d1}). 
As a corollary of our characterization, we characterize bases that satisfy \eqref{e1} w.const $\mathbf C = 1$ in the same manner as Berasategui et al. characterized strong partially greedy bases w.const $1$ \cite{BBL}.

\begin{cor}\label{ca1}
A basis is $1$-RPG2 w.const $1$ if and only if it is RPG w.const $1$.
\end{cor}
For the full statement, see Corollary \ref{ca2}.

\section{A characterization of RPG and PG bases}

We recall some well-known results that will be used in due course.

\begin{defi}[Konyagin and Temlyakov \cite{KT1}]\label{d3}\normalfont
A basis $\mathcal{B}$ is \textbf{quasi-greedy} w.const $\mathbf C > 0$ if
$$\|P_\Lambda(x)\|\ \leqslant\ \mathbf C\|x\|,\forall x\in\mathbb{X}, m\in\mathbb{N}, \forall \Lambda\in G(x, m).$$
The least such $\mathbf C$ is denoted by $\mathbf C_q$, called the quasi-greedy constant. Also when $\mathcal{B}$ is quasi-greedy, let $\mathbf C_\ell$ be the least constant such that $$\|x-P_\Lambda(x)\|\ \leqslant\ \mathbf \mathbf \mathbf C_\ell\|x\|,\forall x\in\mathbb{X}, m\in\mathbb{N}, \forall \Lambda\in G(x, m).$$
We call $\mathbf C_\ell$ the suppression quasi-greedy constant.
\end{defi}

A sequence $\varepsilon = (\varepsilon_n)_n$ is called a sign if $\varepsilon_n\in\mathbb{K}$ and $|\varepsilon_n| = 1$. For $x\in \mathbb{X}$ and a finite set $A\subset\mathbb{N}$, let
$$1_A \ =\  \sum_{n\in A} e_n, 1_{\varepsilon A} \ =\ \sum_{n\in A}\varepsilon_n e_n, \mbox{ and }P_{A^c}(x)\ =\ x - P_A(x).$$

Two important properties of quasi-greedy bases are the \textbf{UL property} and the uniform boundedness of the truncation operators. 
\begin{enumerate}
    \item[i)] UL property: for all finite $A\subset\mathbb{N}$ and scalars $(a_n)$, we have
$$\frac{1}{2\mathbf C_q}\min |a_n|\|1_A\|\ \leqslant\ \left\|\sum_{n\in A}a_ne_n\right\|\ \leqslant\ 2\mathbf C_q\max|a_n|\|1_A\|,$$
which was first proved in \cite{DKKT}. 

    \item[ii)] the uniform boundedness of the truncation operator: for each $\alpha > 0$, we define the truncation function $T_\alpha$ as follows: for  $b\in\mathbb{K}$,
$$T_{\alpha}(b)\ =\ \begin{cases}\sgn(b)\alpha, &\mbox{ if }|b| > \alpha,\\ b, &\mbox{ if }|b|\leqslant \alpha.\end{cases}$$
We define the truncation operator $T_\alpha: \mathbb{X}\rightarrow \mathbb{X}$ as
$$T_{\alpha}(x)\ =\ \sum_{n=1}^\infty T_\alpha(e_n^*(x))e_n \ =\ \alpha 1_{\varepsilon \Lambda_{\alpha}(x)}+ P_{\Lambda_\alpha^c(x)}(x),$$
where $\Lambda_\alpha(x) = \{n: |e_n^*(x)| > \alpha\}$ and $\varepsilon_n = \sgn(e_n^*(x))$ for all $n\in \Lambda_\alpha(x)$.

\begin{thm}\label{bto}\cite[Lemma 2.5]{BBG} Let $\mathcal{B}$ be suppression quasi-greedy w.const $\mathbf C_\ell$. Then for any $\alpha > 0$, $\|T_\alpha\|\leqslant \mathbf{C}_\ell$.
\end{thm}
\end{enumerate}

Now we recall a characterization of RPG bases due to Dilworth and Khurana \cite{DK}. 

\begin{defi}\normalfont
A basis $\mathcal{B}$ is said to be \textbf{reverse conservative} w.const $\mathbf C > 0$ if we have
$$\|1_A\|\ \leqslant\ \mathbf C\|1_B\|, $$
for all finite sets $A, B\subset\mathbb{N}$ with $|A|\leqslant |B|$ and $B < A$. In this case, the least constant $\mathbf C$ is denoted by $\Delta_{rc}$. 
\end{defi}

\begin{thm}\label{m1}\cite[Theorem 2.7]{DK}
A basis is RPG if and only if it is reverse conservative and quasi-greedy. 
\end{thm}

We are ready to prove our first main result.

\begin{proof}[Proof of Theorem \ref{m2}]
Assume \eqref{e1}. By Theorem \ref{m1}, we need to show that $\mathcal{B}$ is both quasi-greedy and reverse conservative. First, we observe that $\mathcal{B}$ is suppression quasi-greedy w.const $\mathbf C$ since by definition, $\widecheck{\sigma}^{\Lambda}_m(x) \leqslant \|x\|$ for all $x\in\mathbb{X}$, $m\in\mathbb{N}$, and $\Lambda\in G(x, m)$. Next, we show that $\mathcal{B}$ is reverse conservative. Let $A, B\subset\mathbb{N}$ be finite with $B < A$ and $|A|\leqslant |B|$. If $A = \emptyset$, there is nothing to prove. Suppose $A\neq \emptyset$. Define $x = 1_A + 1_B + 1_D$, where $D = [\min A, \max A]\backslash A$. Since $B\cup D$ is a greedy set of $x$ and $A\cup D$ is an interval of length $|A\cup D|\leqslant |B\cup D|$ with $B\cup D < \max(A\cup D)$, we have
$$\|1_A\|\ =\ \|x - P_{B\cup D}(x)\|\ \leqslant\ \mathbf C\widecheck{\sigma}^{B\cup D}_{|B\cup D|}(x)\ \leqslant\ \mathbf C\|x-P_{A\cup D}(x)\|\ =\ \mathbf C\|1_B\|.$$

Now we assume that $\mathcal{B}$ is RPG. According to Theorem \ref{m1}, $\mathcal{B}$ is quasi-greedy w.const $\mathbf C_q$ and reverse conservative w.const $\Delta_{rc}$ for some $\mathbf C_q, \Delta_{rc}> 0$. Fix $x\in \mathbb{X}$, $m\in\mathbb{N}$, $\Lambda\in G(x,m)$, and a nonempty interval $I$ with $|I|\leqslant m$, $\Lambda\leqslant \max I$. We shall show that $$\|x-P_{\Lambda}(x)\|\ \leqslant\ (1+\mathbf C_q + 4\mathbf C^3_q\Delta_{rc})\|x-P_I(x)\|.$$
Write 
\begin{equation}\label{e2}\|x-P_{\Lambda}(x)\|\ \leqslant\ \|x-P_I(x)\| + \|P_{\Lambda\backslash I}(x)\| + \|P_{I\backslash \Lambda}(x)\|.\end{equation}
Since $\Lambda\backslash I$ is a greedy set of $x - P_I(x)$, we get
\begin{equation}\label{e3}\|P_{\Lambda\backslash I}(x)\|\ \leqslant\ \mathbf C_q\|x-P_I(x)\|.\end{equation}
Let us estimate $\|P_{I\backslash \Lambda}(x)\|$. Note that 
$\Lambda\backslash I < I\backslash \Lambda$. Otherwise, there exists $a\in \Lambda\backslash I$, $b\in I\backslash \Lambda$ such that $a\geqslant b$. Then $a\geqslant \min I$ and $a\leqslant \max \Lambda\leqslant \max I$. Hence, $a\in I$, which contradicts $a\in \Lambda\backslash I$. Furthermore, $|I\backslash \Lambda|\leqslant |\Lambda\backslash I|$ because $|I|\leqslant |\Lambda|$. Since $\mathcal{B}$ is $\Delta_{rc}$-reverse conservative, we get 
$$\|1_{I\backslash \Lambda}\|\ \leqslant\ \Delta_{rc}\|1_{\Lambda\backslash I}\|.$$
We have
\begin{align}\label{e4}
    \left\|\sum_{n\in I\backslash \Lambda}e_n^*(x)e_n\right\|&\ \leqslant\ \max_{n\in I\backslash \Lambda}|e_n^*(x)| \sup_{\varepsilon}\left\|1_{\varepsilon (I\backslash \Lambda)}\right\|\mbox{ by convexity}\nonumber\\
    &\ \leqslant\ 2\mathbf C_q \min_{n\in \Lambda\backslash I} |e_n^*(x)|\left\|1_{ I\backslash \Lambda}\right\|\mbox{ by the UL property}\nonumber\\
    &\ \leqslant\ 2\mathbf C_q\Delta_{rc} \min_{n\in \Lambda\backslash I} |e_n^*(x)|\left\|1_{ \Lambda\backslash I}\right\|\mbox{ by reverse conservativeness}\nonumber\\
    &\ \leqslant\ 4\mathbf C^2_q\Delta_{rc}\|P_{\Lambda\backslash I}(x)\|\mbox{ by the UL property}.
\end{align}
From \eqref{e2}, \eqref{e3}, and \eqref{e4}, we obtain
$$\|x - P_{A}(x)\|\ \leqslant\ (1+\mathbf C_q + 4\mathbf C^3_q\Delta_{rc})\|x - P_I(x)\|,$$
as desired.
\end{proof}

The proof of Theorem \ref{ma1} is the same as the proof of Theorem \ref{m2} with obvious modifications. The proof is left for interested readers. 

\section{Larger greedy sums for RPG bases}

Throughout this section, let $\lambda$ be a real number at least $1$. First, we define the notion of reverse partial symmetry for largest coefficients (RPSLC). For a finite set $A\subset\mathbb{N}$, let $s(A) = 0$ if $A= \emptyset$; otherwise, $s(A) = \max A - \min A + 1$. We can think of $s(A)$ as $|\mbox{co}(A)|$, where $\mbox{co}(A)$ is the integer convex hull of the set $A$. For a collection of sets $(A_i)_{i\in I}$, we write $\sqcup_i A_i$ to mean $A_j\cap A_k = \emptyset$ for any $j, k\in I$. Finally, $\|x\|_\infty:= \max_{n}|e_n^*(x)|$.

\begin{defi}\normalfont
A vector $x\in\mathbb{X}$ is said to surround a finite set $A\subset\mathbb{N}$ if either $A = \emptyset$ or $\supp(x)\cap [\min A, \max A] = \emptyset$.
\end{defi}

\begin{defi}\label{d1}\normalfont
A basis is said to be $\lambda$-RPSLC if there exists a constant $\mathbf C \geqslant 1$ such that
$$\|x + 1_{\varepsilon A}\|\ \leqslant\ \mathbf C\|x+1_{\delta B}\|,$$
for all $x\in\mathbb{X}$ with $\|x\|_\infty\leqslant 1$, for all signs $\varepsilon, \delta$, and for all finite sets $A, B\subset\mathbb{N}$ such that $(\lambda - 1)s(A) + |A|\leqslant |B|$, $B\sqcup \supp(x)$, $B < A$, and $x$ surrounds $A$. In this case, the least such $\mathbf C$ is denoted by $\Delta_{\lambda, rpl}$.
\end{defi}

\begin{rek}\normalfont
While the definition of $\lambda$-RPSLC may seem unnatural and technical, we shall use it in proving Corollary \ref{ca2} and eventually Theorem \ref{m8}, both of which are relevant to the literature on greedy-type bases. 
\end{rek}

We have an easy but useful reformulation of $\lambda$-RPSLC.

\begin{prop}\label{p1}
A basis $\mathcal{B}$ is $\lambda$-RPSLC w.const $\Delta_{\lambda, rpl}$ if and only if 
\begin{equation}\label{e6}\|x\|\ \leqslant\ \Delta_{\lambda, rpl}\|x - P_A(x) + 1_{\varepsilon B}\|,\end{equation}
for all $x\in\mathbb{X}$ with $\|x\|_\infty\leqslant 1$, for all sign $\varepsilon$, and for all finite sets $A, B\subset\mathbb{N}$
such that $(\lambda - 1)s(A) + |A|\leqslant |B|$, $B\sqcup \supp(x-P_A(x))$, $B < A$, and $x-P_A(x)$ surrounds $A$. 
\end{prop}

\begin{proof}
Assume that $\mathcal{B}$ is $\lambda$-RPSLC w.const $\Delta_{\lambda, rpl}$. Let $x, A, B, \varepsilon$ be chosen as in \eqref{e6}. We have
\begin{align*}
    \|x\|\ =\ \left\|x - P_A(x) + \sum_{n\in A}e_n^*(x)e_n\right\|&\ \leqslant\ \sup_{\delta}\left\|x-P_A(x) + 1_{\delta A}\right\|\\
    &\ \leqslant\ \Delta_{\lambda, rpl}\|x-P_A(x) + 1_{\varepsilon B}\|.
\end{align*}
Next, assume that $\mathcal{B}$ satisfies \eqref{e6}. Let $x, A, B, \varepsilon, \delta$ be chosen as in Definition \ref{d1}. Let $y = x+ 1_{\varepsilon A}$. By \eqref{e6}, 
\begin{align*}
    \|x+1_{\varepsilon A}\|\ =\ \|y\|\ \leqslant\ \Delta_{\lambda, rpl}\|y - P_A(y) + 1_{\delta B}\|\ =\ \Delta_{\lambda, rpl}\|x+1_{\delta B}\|.
\end{align*}
This completes our proof. 
\end{proof}

The next theorem characterizes $\lambda$-RPG2 bases. 

\begin{thm}\label{m3}
A basis $\mathcal{B}$ is $\lambda$-RPG2 if and only it is quasi-greedy and $\lambda$-RPSLC.
\end{thm}

The next theorem, which describes the relation between constants of the $\lambda$-RPG2 property and RPSLC, facilitates the proof of Theorem \ref{m3}.

\begin{thm}(Analog of \cite[Theorem 4.1]{C1})\label{m4}
Let $\mathcal{B}$ be suppression quasi-greedy w.const $\mathbf{C}_\ell$. The following hold.
\begin{itemize}
    \item [i)] If $\mathcal{B}$ is $1$-RPG2 w.const $\mathbf C_{1, rp2}$, then $\mathcal{B}$ is $1$-RPSLC w.const $\mathbf C_{1, rp2}$.
    \item [ii)] If $\mathcal{B}$ is $\lambda$-RPG2 w.const $\mathbf C_{\lambda, rp2}$, then $\mathcal{B}$ is $\lambda$-RPSLC w.const $\mathbf C_\ell\mathbf C_{\lambda, rp2}$.
    \item [iii)] If $\mathcal{B}$ is $\lambda$-RPSLC w.const $\Delta_{\lambda, rpl}$, then $\mathcal{B}$ is $\lambda$-RPG2 w.const $\mathbf C_\ell\Delta_{\lambda, rpl}$.
\end{itemize}
\end{thm}

\begin{proof}
i) Let $x, A, B, \varepsilon, \delta$ be as in Definition \ref{d1}. 
We need to show that 
$$\|x+1_{\varepsilon A}\|\ \leqslant\ \mathbf C_{1, rp2}\|x+1_{\delta B}\|.$$
If $A = \emptyset$, then 
$$\|x+1_{\varepsilon A}\|\ =\ \|x\|\ \leqslant\ \mathbf C_{\ell}\|x+1_{\delta B}\|\ \leqslant\ \mathbf C_{1, rp2}\|x+1_{\delta B}\|,$$
where the last inequality is due to Definition \ref{bdf}.
If $A \neq \emptyset$, form $y = x + 1_{\varepsilon A} + 1_{\delta B} + 1_D$, where $D = [\min A, \max A]\backslash A$. We have
$$|D\cup A|\ \leqslant\ |D\cup B|\mbox{ and }D\cup B < \max(D\cup A).$$
Hence, we obtain
\begin{align*}\|x+1_{\varepsilon A}\|\ =\ \|y - P_{B\cup D}(y)\|\ \leqslant\ \mathbf C_{1, rp2}\widecheck{\sigma}^{B\cup D}_{|B\cup D|}(y)&\ \leqslant\ \mathbf C_{1, rp2}\|y- P_{A\cup D}(y)\|\\
&\ =\ \mathbf C_{1, rp2}\|x + 1_{\delta B}\|,
\end{align*}
as desired.

ii) Let $x, A, B, \varepsilon, \delta$ be as in Definition \ref{d1}. If $A = \emptyset$, then by the proof of item i), we are done. If $A \neq\emptyset$, form $y = x + 1_{\varepsilon A} + 1_{\delta B} + 1_D$, where $D = [\min A, \max A]\backslash A$. Observe that $B\cup D$ is a greedy set of $y$, and 
$$|B\cup D| \ =\ |B| + |D|\ \geqslant\ (\lambda-1)s(A) + |A| + (s(A) - |A|)\ =\ \lambda s(A).$$
Choose $\Lambda\subset B\cup D$ such that $|\Lambda| = \lceil \lambda s(A)\rceil$. Clearly, 
$$|A\cup D|\ =\ s(A) \mbox{ and }\Lambda\ <\ \max (A\cup D).$$
We obtain
\begin{align*}\|x+1_{\varepsilon A}\|\ =\ \|y - P_{B\cup D}(y)\|\ \leqslant\ \mathbf C_\ell\|y-P_{\Lambda}(y)\|&\ \leqslant\ \mathbf C_\ell\mathbf C_{\lambda, rp2}\widecheck{\sigma}^{\Lambda}_{s(A)}(y)\\
&\ \leqslant\  \mathbf C_\ell\mathbf C_{\lambda, rp2}\|y - P_{A\cup D}(y)\|\\
&\ =\ \mathbf C_\ell\mathbf C_{\lambda, rp2}\|x+1_{\delta B}\|.
\end{align*}
Therefore, $\mathcal{B}$ is $\lambda$-RPSLC w.const $\mathbf C_\ell\mathbf C_{\lambda, rp2}$.

iii) Assume that $\mathcal{B}$ is suppression quasi-greedy w.const $\mathbf C_\ell$ and $\lambda$-RPSLC w.const $\Delta_{\lambda, rpl}$. Let $x\in\mathbb{X}$, $m\in\mathbb{N}$, $\Lambda\in G(x, \lceil \lambda m\rceil)$, and a nonempty interval $I$ with $|I|\leqslant m$ and $\Lambda\leqslant \max I$. We need to show that 
$$\|x-P_{\Lambda}(x)\|\ \leqslant\ \mathbf C_\ell \Delta_{\lambda, rpl}\|x - P_I(x)\|.$$
If $\alpha := \min_{n\in\Lambda} |e_n^*(x)|$, then $\|x-P_{\Lambda}(x)\|_\infty\leqslant \alpha$. We have
\begin{align*}
    |\Lambda\backslash I| \ =\ |\Lambda| - |\Lambda\cap I|\ \geqslant\ \lambda m - |\Lambda\cap I| &\ =\ \lambda m + (|I| - |\Lambda\cap I|) - |I|\\
    &\ \geqslant\ \lambda m + |I\backslash \Lambda| - m\\
    &\ =\ (\lambda - 1)m + |I\backslash \Lambda|\\
    &\ \geqslant\ (\lambda-1)s(I\backslash \Lambda) + |I\backslash \Lambda|,
\end{align*}
and 
$$\Lambda\backslash I\ <\ I\backslash \Lambda\mbox{ and }(\Lambda \backslash I)\sqcup \supp(x-P_{\Lambda}(x)-P_{I\backslash \Lambda}(x)).$$
Furthermore, $x-P_{\Lambda}(x)-P_{I\backslash \Lambda}(x)$ surrounds $I\backslash \Lambda$.
Setting $\varepsilon = (\sgn(e_n^*(x))$, we can apply Proposition \ref{p1} to obtain
\begin{align*}
    \|x-P_{\Lambda}(x)\|&\ \leqslant\ \Delta_{\lambda, rpl}\|x-P_{\Lambda}(x) - P_{I\backslash \Lambda}(x) + \alpha 1_{\varepsilon \Lambda\backslash I}\| \\
    &\ =\ \Delta_{\lambda, rpl}\|T_{\alpha}(x-P_{\Lambda}(x) - P_{I\backslash \Lambda}(x) + P_{\Lambda\backslash I}(x))\|\\
    &\ \leqslant\ \mathbf C_\ell \Delta_{\lambda, rpl}\|x-P_I(x)\| \mbox{ by Theorem \ref{bto}.}
\end{align*}
We finish the case when $I \neq \emptyset$. If $I  = \emptyset$, then we simply have $\|x-P_{\Lambda}(x)\|\leqslant \mathbf C_\ell\|x\| = \mathbf C_\ell\|x-P_I(x)\|$. 
\end{proof}

\begin{proof}[Proof of Theorem \ref{m3}]
By Theorem \ref{m4}, we need only to show that a $\lambda$-RPG2 basis is quasi-greedy. Indeed, a $\lambda$-RPG2 basis satisfies \eqref{e5}, which implies that
$$\|x-P_{\Lambda}(x)\|\ \leqslant\ \mathbf C\|x\|, \forall x\in\mathbb{X}, \forall m\in \mathbb{N},\forall \Lambda\in G(x, \lceil \lambda m\rceil).$$
By \cite[Proposition 4.1]{O}, we know that $\mathcal{B}$ is quasi-greedy. 
\end{proof}

\begin{defi}\normalfont
A basis is said to be $\lambda$-reverse conservative of type 2\footnote{To distinguish from \cite[Definition 5.6]{C1}.} if for some constant $\mathbf C > 0$, we have
$$\|1_A\|\ \leqslant\ \mathbf C\|1_B\|, $$
for all finite sets $A, B\subset\mathbb{N}$ with $(\lambda-1)s(A) + |A|\leqslant |B|$ and $B < A$. In this case, the least constant $\mathbf C$ is denoted by $\Delta_{\lambda, rc}$. 
\end{defi}

\begin{prop}\label{p2}
Let $\mathcal{B}$ be a quasi-greedy basis. Then $\mathcal{B}$ is $\lambda$-reverse conservative of type 2 if and only if it is $\lambda$-RPSLC. 
\end{prop}

\begin{proof}
Setting $x = 0$ and $\varepsilon \equiv \delta \equiv 1$ in Definition \ref{d1}, we see that a $\lambda$-RPSLC basis is $\lambda$-reverse conservative of type 2. Assume that $\mathcal{B}$ is $\lambda$-reverse conservative of type 2 w.const $\Delta_{\lambda, rc}$, suppression quasi-greedy w.const $\mathbf C_\ell$, and quasi-greedy w.const $\mathbf C_q$. Let $x, A, B, \varepsilon, \delta$ be chosen as in Definition \ref{d1}. We have
$$\|1_{\varepsilon A}\|\ \stackrel{\mbox{UL}}{\leqslant}\ 2\mathbf C_q\|1_{A}\|\ \leqslant\ 2\mathbf C_q\Delta_{\lambda, rc}\|1_B\|\ \stackrel{\mbox{UL}}{\leqslant}\ 4\mathbf C_q^2\Delta_{\lambda, rc}\|1_{\delta B}\|\ \leqslant\ 4\mathbf C_q^3\Delta_{\lambda, rc}\|x+1_{\delta B}\|.$$
Furthermore, 
$$\|x\|\ \leqslant\ \mathbf C_\ell \|x+1_{\delta B}\|.$$
Therefore,
$$\|x+1_{\varepsilon A}\|\ \leqslant\ \|x\| + \|1_{\varepsilon A}\|\ \leqslant\ (4\mathbf C_q^3\Delta_{\lambda, rc} + \mathbf C_\ell)\|x+1_{\delta B}\|.$$
This completes our proof. 
\end{proof}

\begin{thm}\label{m9}
Let $\mathcal{B}$ be a basis. The following are equivalent
\begin{enumerate}
    \item[i)] $\mathcal{B}$ is $\lambda$-RPG2.
    \item[ii)] $\mathcal{B}$ is quasi-greedy and $\lambda$-RPSLC.
    \item[iii)] $\mathcal{B}$ is quasi-greedy and $\lambda$-reverse conservative of type 2. 
\end{enumerate}
\end{thm}
\begin{proof}
The equivalence between i) and ii) is due to Theorem \ref{m3}, and the equivalence between ii) and iii) is due to Proposition \ref{p2}.
\end{proof}

The problem of characterizing $1$-greedy-type bases has been of great interest as can be seen in \cite{AA0, AA, AW, BBL, DK}.  
As a corollary of Theorem \ref{m4}, let us characterize bases that are $1$-RPG2 w.const $1$.

\begin{thm}\label{m6}
A basis $\mathcal{B}$ is $1$-RPG2 w.const $1$ if and only if $\mathcal{B}$ is $1$-RPSLC w.const $1$. 
\end{thm}

\begin{proof}
By Theorem \ref{m4} items i) and iii), it suffices to show that if a basis $\mathcal B$ is $1$-RPSLC w.const $1$, then it is suppression quasi-greedy w.const $1$. This follows immediately from Definition \ref{d1} by setting $A = \emptyset$ to have
$$\|x\|\ \leqslant\ \|x + e_k\|,$$
for all $x\in \mathbb{X}$ with $\|x\|_\infty\leqslant 1$ and for all $k\notin \supp(x)$. By induction, $\mathcal{B}$ is suppression quasi-greedy w.const $1$. 
\end{proof}

In the spirit of \cite[Proposition 4.2]{BBL}, we offer yet another characterization of $1$-RPG2 bases w.const $1$.

\begin{thm}\label{m7}
A basis $\mathcal{B}$ is $1$-RPG2 w.const $1$ if and only if $\mathcal{B}$ satisfies simultaneously two following conditions
\begin{enumerate}
    \item[i)] for all $x\in\mathbb{X}$ with $\|x\|_\infty\leqslant 1$ and for all $k\notin \supp(x)$, we have
    \begin{equation}
        \label{e15}\|x\|\ \leqslant\ \|x+e_k\|.
    \end{equation}
    \item[ii)] for $x\in\mathbb{X}$ with $\|x\|_\infty\leqslant 1$, for $s, t\in\mathbb{K}$ with $|s| = |t| = 1$, and for $j < k$, both of which are not in $\supp(x)$, we have
    \begin{equation}\label{e16}\|x+te_k\|\ \leqslant\ \|x+se_j\|.\end{equation}
\end{enumerate}
\end{thm}

\begin{proof}
Due to Theorem \ref{m6}, it suffices to show that a basis is $1$-RPSLC w.const $1$ if and only if it satisfies both \eqref{e15} and \eqref{e16}. It follows immediately from Definition \ref{d1} that a $1$-RPSLC basis w.const $1$ must satisfy \eqref{e15} and \eqref{e16}. Conversely, suppose that $\mathcal{B}$ satisfies both \eqref{e15} and \eqref{e16}. Choose $x, A, B, \varepsilon, \delta$ as in Definition \ref{d1}. We show that 
\begin{equation}\label{e17}\|x+1_{\varepsilon A}\|\ \leqslant\ \|x + 1_{\delta B}\|\end{equation}
inductively on $|B|$. 

Base case: $|B| = 1$. If $A = \emptyset$, then \eqref{e15} implies \eqref{e17}; if $|A| = 1$, then \eqref{e16} implies \eqref{e17}.

Inductive hypothesis (I.H.): assume that for some $\ell\in\mathbb{N}$, \eqref{e17} holds for $|B|\leqslant \ell$. We show that \eqref{e17} holds for $|B| = \ell+1$. If $A = \emptyset$, then we use \eqref{e15} inductively to obtain \eqref{e17}. For $|A| \geqslant 1$, let $p = \max A$, $q\in B$, $A' = A\backslash \{p\}$, and $B' = B\backslash \{q\}$. We have
\begin{align*}\|x+1_{\varepsilon A}\|\ =\ \|(x + \varepsilon_p e_p) + 1_{\varepsilon A'}\|&\ \leqslant\ \|(x+\varepsilon_p e_p) + 1_{\delta B'}\|\mbox{ by I.H.}\\
&\ \leqslant\ \|(x+\delta_q e_q) + 1_{\delta B'}\|\mbox{ by \eqref{e16}}\\
&\ =\ \|x+1_{\delta B}\|.\end{align*}
This shows that $\mathcal{B}$ is $1$-RPSLC w.const $1$. 
\end{proof}

\begin{cor}\label{ca2}
The following are equivalent
\begin{enumerate}
    \item [i)] $(e_n)_n$ is $1$-RPG2 w.const $1$.
    \item [ii)] $(e_n)_n$ is $1$-RPSLC w.const $1$. 
    \item [iii)] $(e_n)_n$ satisfies \eqref{e15} and \eqref{e16}.
    \item [iv)] $(e_n)_n$ is RPG w.const $1$. 
\end{enumerate}
\end{cor}

\begin{proof}
The equivalence among i), ii), and iii) is due to Theorems \ref{m6} and \ref{m7}. That iii) is the same as iv) is due to \cite[Theorem 3.4]{DK}.
\end{proof}

\section{The $\lambda$-RPG2 property is weaker than the RPG property}
The goal of this section is to prove Theorem \ref{m8}.

\begin{proof}
i) If $\mathcal{B}$ is RPG, then it is quasi-greedy and reverse conservative. By definition, a reverse conservative basis is $\lambda$-reverse conservative of type 2. By Theorem \ref{m9}, $\mathcal{B}$ is $\lambda$-RPG2. 

ii) For each $\lambda > 1$, we now construct an unconditional basis $\mathcal{B}$ that is $\lambda$-RPG2 but is not RPG. Let $D = \{2^n\,:\, n\in \mathbb{N}_0\}, u_n = \frac{1}{\sqrt{n}}$, and $v_n = \frac{1}{n}$ for all $n\geqslant 1$.
Let $\mathbb{X}$ be the completion of $c_{00}$ with respect to the following norm: for $x = (x_1, x_2, \ldots)$, define
$$\|x\|\ =\ \sup_{\pi, \pi'}\left(\sum_{n\in D}u_{\pi(n)} |x_{n}| + \sum_{n\notin D}v_{\pi'(n)}|x_n|\right),$$
where $\pi: D\rightarrow \mathbb{N}$ and $\pi': \mathbb{N}\backslash D\rightarrow \mathbb{N}$ are bijections. 
Clearly, the canonical basis $\mathcal{B}$ is an $1$-unconditional normalized basis of $\mathbb{X}$. 

\begin{claim}
The basis $\mathcal{B}$ is not reverse conservative and thus, is not RPG. 
\end{claim}
\begin{proof}
We use the notation $\sim$ to indicate the order of a number.
For each $N\in\mathbb{N}$, let $A_N = \{2^{2N+1}, 2^{2N+2}, \ldots, 2^{3N}\}$ and $B_N = \{3, 3^2, \ldots, 3^N\}$. Then $|A_N| = |B_N| = N$ and $B_N < A_N$. However,
$$\|1_{A_N}\|\ =\ \sum_{n=1}^N\frac{1}{\sqrt{n}}\ \sim\ \sqrt{N} \mbox{ and }\|1_{B_N}\|\ =\ \sum_{n=1}^N\frac{1}{n} \ \sim\ \ln N,$$
which imply that $\|1_{A_N}\|/\|1_{B_N}\|\rightarrow\infty$ as $N\rightarrow\infty$. Therefore, $\mathcal{B}$ is not reverse conservative. 
\end{proof}

\begin{claim}
The basis $\mathcal{B}$ is $\lambda$-reverse conservative of type 2 and thus, is $\lambda$-RPG2.
\end{claim}
\begin{proof}
Pick nonempty, finite sets $A, B\subset\mathbb{N}$ with $(\lambda-1)s(A) + |A|\leqslant |B|$ and $B < A$. Then
$$\|1_A\|\ =\ \sum_{n=1}^{|A\cap D|}\frac{1}{\sqrt{n}} + \sum_{n=1}^{|A\backslash D|}\frac{1}{n}\mbox{ and }\|1_B\|\ =\ \sum_{n=1}^{|B\cap D|}\frac{1}{\sqrt{n}} + \sum_{n=1}^{|B\backslash D|}\frac{1}{n}.$$

\begin{claim}\label{cll1}
We have $\|1_B\|\gtrsim \ln |B|$.
\end{claim}

\begin{proof}
If $B\backslash D = \emptyset$, then $B\cap D = B$ and $\|1_B\| = \sum_{n=1}^{|B|}1/\sqrt{n} \sim \sqrt{|B|}$. 
Similarly, if $B\cap D = \emptyset$, then $B\backslash D = B$ and $\|1_B\| = \sum_{n=1}^{|B|}1/n \sim \ln{|B|}$.
Finally, if $B\backslash D\neq \emptyset$ and $B\cap D\neq \emptyset$, then 
$$\|1_B\|\ \gtrsim\ \ln{|B\cap D|} + \ln|B\backslash D|\ =\ \ln(|B\cap D|\cdot|B\backslash D|)\ \geqslant\ \ln(|B|-1)\ \gtrsim\ \ln|B|.$$
\end{proof}

We proceed by case analysis.

Case 1: $\sum_{n=1}^{|A\cap D|}\frac{1}{\sqrt{n}}\ \leqslant\ \sum_{n=1}^{|A\backslash D|}\frac{1}{n}$. We get
$$\|1_A\|\ \leqslant\ 2\sum_{n=1}^{|A\backslash D|}\frac{1}{n}\ \leqslant\ 2\sum_{n=1}^{|B|}\frac{1}{n}\ \leqslant\ 2\left(\sum_{n=1}^{|B\cap D|}\frac{1}{\sqrt{n}} + \sum_{n=1}^{|B\backslash D|}\frac{1}{n}\right)\ =\ 2\|1_B\|.$$

Case 2: $\sum_{n=1}^{|A\cap D|}\frac{1}{\sqrt{n}}\ >\ \sum_{n=1}^{|A\backslash D|}\frac{1}{n}$. Then $\|1_A\|\sim \sqrt{|A\cap D|}$.
If $|B\cap D|\geqslant |A\cap D|$, then we are done because $$\|1_B\| \ \geqslant\ \sum_{n=1}^{|B\cap D|}\frac{1}{\sqrt{n}}\ \sim\ \sqrt{|B\cap D|}\ \geqslant\ \sqrt{|A\cap D|}\ \sim\ \|1_A\|.$$
Suppose that $|B\cap D| < |A\cap D|$. Let $N = |A\cap D|$ and write $A\cap D = \{2^{k_1}, 2^{k_2}, \ldots, 2^{k_N}\}$ to obtain $\|1_A\|\sim \sqrt{N}$. 

\begin{enumerate}
    \item[i)] Case 2.1: $N = 1$. Then $\sum_{n=1}^{|A\cap D|}\frac{1}{\sqrt{n}}\ >\ \sum_{n=1}^{|A\backslash D|}\frac{1}{n}$ implies that $|A\backslash D| = 0$ and $|A| = 1$. Clearly, $\|1_A\|\leqslant \|1_B\|$.
    \item[ii)] Case 2.2: $N\geqslant 2$. We have
$$|B|\ \geqslant\ (\lambda-1)s(A) \ \geqslant\ (\lambda-1)(2^{k_N}-2^{k_1}+1)\ \geqslant\ (\lambda-1)2^{k_N-1}.$$
Hence,
$$\ln|B|\ \geqslant\ \ln(\lambda-1) + (N-1)\ln 2.$$
If $\ln(\lambda-1) + 0.5(N-1)\geqslant 0$, then by Claim \ref{cll1},
$$\|1_B\| \ \gtrsim\ \ln |B|\ \geqslant\ (\ln 2 -0.5)(N-1) \ \gtrsim \ \sqrt{N} \ \sim\ \|1_A\|.$$
If $\ln(\lambda-1) + 0.5(N-1) < 0$, then $N < 1 - 2\ln (\lambda-1)$. In this case,
$$\|1_A\|\ \sim\ \sqrt{N}\ <\ \sqrt{1 - 2\ln (\lambda-1)}\ \leqslant\  \sqrt{1 - 2\ln (\lambda-1)}\|1_B\|.$$
\end{enumerate}
We have shown that in all cases, there exists a constant $\mathbf C = \mathbf C(\lambda)$ such that $\|1_A\| \leqslant \mathbf C\|1_B\|$. Therefore, $\mathcal{B}$ is $\lambda$-reverse consecutive of type 2.
\end{proof}

We conclude that our basis $\mathcal{B}$ is $1$-unconditional and $\lambda$-RPG2 but is not RPG. 
\end{proof}

%%%%%%%%%%%%%%%%%%%%%%%%%%%%%%%%%%%%%%%%%%%%%%%%%%%%%%%%%%%%%%%%%%%%%%%%%%%%%%%%%%%%%%%%%%%%%%%%%%%%%%%%%%%%%%%%%%%%%%%%%%%%%%%%%%%%%%%%%%%%%%%%%%%%%%%%%%%%%%%%%%%%%%%%%%%%%%%%%%%%%%%%

%%%%%%%%%%%%%%%%%%%%%%%%%%%%%%%%%%%%%%%%%%%%%%%%%%%%%%%%%%%%%%%%%%%%%%%%%%%%%%%%%%%%%%%%%%%%%%%%%%%%%%%%%%%%%%%%%%%%%%%%%%%%%%%%%%%%%%%%%%%%%%%%%%%%%%%%%%%%%%%%%%%%%%%%%%%%%%%%%%%%%%%%%%%%%%%%%%%%%%%%%%%%%%%%%%%%%%%%%%%%%%%%%%%%%%%%%%%%%%%%%%%%%%%%%%%%%%%%%%%%%%%%%%%%%%%%%%%%%%%%%%%%%%%%%%%%%%%%%%%%%%%%%%%%%%%%%%%%%%%%%%%%%%%%%%%%%%%%%%%%%%%%%%%%%%%%%%%%%%%%%%%%%%%%%%%%%%%%%%%%%%%%%%%%%%%%%%%%%%%%%%%%%%%%%%%%%%%%%%%%%%%%%%%%%%%%%%%%%%%%%%%%%%%%%%%%%%%%%%%%%%%%%%%%%%%%%%%%%%%%%%%%%%%%%%%%%%%%%%%%%%%%%%%%%%%%%%%%%%%%%%%%%%%%%%%%%%%%%%%%%%%%
%%%%%%%%%%%%%%%%%%%%%%%%%%%%%%%%%%%%%%%%%%%%%%%%%%%%%%%%%%%%%%%%%%%%%%%%%%%%%%%%%%%%%%%%%%%%%%%%%%%%%%%%%%%%%%%%%%%%%%%%%%%%%%%%%%%%%%%%%%%%%%%%%%%%%%%%%%%%%%%%%%%%%%%%%%%%%%%%%%%%%%%%%%%%%%%%%%%%%%%%%%%%%%%%%%%%%%%%%%%%%%%%%%%%%%%%%%%%%%%%%%%%%%%%%%%%%%%%%%%%%%%%%%%%%%%%%%%%%%%%%%%%%%%%%%%%%%%%%%%%%%%%%%%%%%%%%%%%%%%%%%%%%%%%%%%%%%%%%%%%%%%%%%%%%%%%%%%%%%%%%%%%%%%%%%%%%%%%%%%%%%%%%%%%%%%%%%%%%%%%%%%%%%%%%%%%%%%%%%%%%%%%%%%%%%%%%%%%%%%%%%%%%%%%%%%%%%%%%%%%%%%%%%%%%%%%%%%%%%%%%%%%%%%%%%%%%%%%%%%%%%%%%%%%%%%%%%%%%%%%%%%%%%%%%%%%%%%%%%%%%%%%

\ \\
\end{document}